\documentclass[reqno]{amsart}
\usepackage{amsmath, amssymb, amsthm, epsfig}
\usepackage{hyperref, latexsym}
\usepackage{url}
\usepackage[mathscr]{euscript}
\usepackage{enumerate}
\usepackage{color}
\usepackage{fullpage} 
\usepackage{setspace}
\usepackage{cancel}

\def\today{\ifcase\month\or
  January\or February\or March\or April\or May\or June\or
  July\or August\or September\or October\or November\or December\fi
  \space\number\day, \number\year}

\DeclareMathOperator{\supp}{\mathrm{supp}}
\def\half{\frac{1}{2}}
\def\supp{\text{supp}}
 \newtheorem{theorem}{Theorem}
  
 \newtheorem{lemma}[theorem]{Lemma}
 \newtheorem{proposition}[theorem]{Proposition}
 
 \theoremstyle{definition}

 \theoremstyle{remark}

 \newtheorem*{prop*}{Proposition}

 \newcommand{\R}{\mathbb{R}}

 \newcommand{\hh}{\tfrac12}
 \newcommand{\ds}{\text{\rm d}s}

 \newcommand{\dt}{\text{\rm d}t}
  \renewcommand{\d}{\text{\rm d}}
 \newcommand{\du}{\text{\rm d}u}

\newcommand{\im}{{\rm Im}\,}
\newcommand{\re}{{\rm Re}\,}

\begin{document}
\title[Large oscillations of the argument of the Riemann zeta-function]{Large oscillations of the argument of the Riemann zeta-function}
\author[Chirre and Kamal]{Andr\'{e}s Chirre and Kamalakshya Mahatab}

\address{Department of Mathematical Sciences, Norwegian University of Science and Technology, NO-7491 Trondheim, Norway}
\email{carlos.a.c.chavez@ntnu.no }
\address{Kamalakshya Mahatab, Department of Mathematics, Chennai Mathematical Institute, H1, SIPCOT IT Park, Siruseri, Kelambakkam 603103, India}
\email{accessing.infinity@gmail.com, \. kmahatab@cmi.ac.in}

\thanks{AC was supported by Grant 275113 of the Research Council of Norway. KM was supported by Grant 227768 of the Research Council of Norway and Project 1309940 of Finnish Academy. A part of this work was carried out when KM was a Leibniz fellow at MFO, Oberwolfach.}



\allowdisplaybreaks
\numberwithin{equation}{section}

\maketitle

\begin{abstract}  
Let $S(t)$ denote the argument of the Riemann zeta-function, defined as
$$
S(t)=\dfrac{1}{\pi}\,\im\log\zeta(1/2+it).
$$
Assuming the Riemann hypothesis, we prove that 
$$
S(t)=\Omega_{\pm}\bigg(\dfrac{\log t\log\log\log t}{\log\log t}\bigg).
$$
 This improves the classical $\Omega$-results of Montgomery \cite[Theorem 2]{M} and matches with the $\Omega$-result obtained by Bondarenko and Seip \cite[Theorem 2]{BS}.
\end{abstract}

\section{Introduction}

In the theory of the Riemann zeta-function, it is important to understand the distribution of its non-trivial zeros
in the critical strip ($0\leq\re(s)\leq1$) and on the critical line ($\re(s)=\half$). The Riemann Hypothesis (RH) asserts that all the zeros in the critical strip lie on the critical line, which we will assume for our main theorem. Let $N(t)$ be the number of zeros of the Riemann zeta-function $\zeta(s)$ in the rectangle $0\leq\re(s)\leq1, 0\leq \im(s)\leq t$, where the zeros with imaginary part $t$ are counted with weight $\hh$. The famous Riemann--von Mangoldt formula (\cite[Chapter 1.4]{iv4}, \cite[Chapter IX]{tit}) asserts that
\[N(t)=\frac{t}{2\pi}\log\left(\frac{t}{2\pi}\right)-\frac{t}{2\pi} + \frac{7}{8}+ S(t)+ O\left(\frac{1}{t}\right).\]
When $t$ is not the ordinate of a non-trivial zero of $\zeta(s)$, the argument function $S(t)$ is defined as
\[S(t)=\frac{1}{\pi}\arg\zeta\left(\half+it\right),\]
where the argument is obtained by continuous variation of $\arg\zeta(s)$ along the polygonal path starting from the point $s=2$ (where $\arg \zeta(2)=0$) and then going first to 
the point $s=2+it$ and then to $s=1/2+it$. If $t$ is the ordinate of a non-trivial zero of $\zeta(s)$, we define
\[S(t)=\lim_{\epsilon\rightarrow0}\frac{1}{2}(S(t+\epsilon)+S(t-\epsilon)).\]
For a detailed exposition on $S(t)$, we refer to \cite{kara} and \cite[Chapter IX]{tit}.

In this paper, we investigate large positive and negative values of $S(t)$. In $1924$, Littlewood \cite[Theorem 11]{L} proved under RH  that
\[S(t)=O\left(\frac{\log t}{\log\log t}\right).\]
The best explicit upper bound know for $S(t)$ under RH is due to Carneiro, Chandee and Milinovich \cite{CCM1} 
\begin{align} \label{14_48_16_20}
|S(t)|\leq\left(\frac{1}{4}+O\left(\frac{\log\log\log t}{\log\log t}\right)\right)\left(\frac{\log t}{\log\log t}\right).
\end{align}
The error term in the above expression was later improved by Carneiro, Milinovich and the first author \cite{CChiM} to $O(1/\log\log t)$. 
\\
On the other hand, in 1977, Montgomery \cite[Theorem]{M} showed that under RH
\begin{equation} \label{21_00_31_12}
S(t)=\Omega_\pm\left(\sqrt{\frac{\log t}{\log\log t}}\right).
\end{equation}
Recently, using the resonance method and assuming RH, Bondarenko and Seip \cite{BS} showed that for each fixed constant $0\leq\beta<1$, there is a constant $c>0$ such that 
\begin{align} \label{eq:bs2}
\max_{T^\beta\leq t\leq T}|S(t)| \geq c\,\sqrt{\frac{\log T \log\log\log T}{\log\log T}}.
\end{align}
Although the estimates \eqref{14_48_16_20} and \eqref{eq:bs2} are the conditionally best  known bounds, the true size of $S(t)$ is perhaps closer to the lower bound. A heuristic argument of Farmer, Gonek and Hughes \cite{farmer} suggests that 
$$
\limsup_{t\to\infty}\dfrac{S(t)}{\sqrt{\log t\log\log t}}=\dfrac{1}{\pi\sqrt{2}}.
$$
We note that from \eqref{eq:bs2} we can not infer if 
\[S(t)\geq c\,\sqrt{\frac{\log t \log\log\log t}{\log\log t}}, \text{ or if } S(t)\leq-c\,\sqrt{\frac{\log t \log\log\log t}{\log\log t}}.\]
We only know that at least one of these bounds is true. Our main goal in this paper is to show that both the bounds are true and hence we improve the $\Omega_\pm$ result of Montgomery \eqref{21_00_31_12}.

\begin{theorem} \label{7:03pm_11_12}
	Assume the Riemann hypothesis. Then, for each fixed constant $0<\beta<1$, there exists a constant $c>0$ such that
	\[\max_{T^\beta\leq t\leq T}\pm S(t) \geq c\,\sqrt{\frac{\log T \log\log\log T}{\log\log T}},\]
for sufficiently large $T$.
\end{theorem}
\subsection{Strategy outline} Our strategy consists of two main ingredients: (i) a new convolution formula for $\log \zeta(s)$ with a suitable parameter $\lambda$, and (ii) the resonator used by Bondarenko and Seip \cite{BS}. We start using a new convolution formula, which is inspired by Montgomery's paper \cite{M}. Assuming RH, Proposition \ref{9_21_4:20pm} shows that for each fixed $0<\lambda<1/2$ and $T$ sufficiently large, we have\footnote{Throughout this paper we will use the notations $\log_2T=\log\log T$ and $\log_3T=\log\log\log T$.}
	\begin{align} \label{11_12_18:33pm}
	\begin{split} 
	\int_{-(\log T)^{3}}^{(\log T)^{3}}\log\zeta\bigg(\dfrac{1}{2}+i(t+u)\bigg)&\bigg(\dfrac{\sin (\lambda u\log _2T)}{u}\bigg)^2( \pm 3 - 2\sin(u\log _2T))\du  \\
	& \,\,\,\,\,\,\,\,= \dfrac{i\pi}{2}\displaystyle\sum_{n=2}^{\infty}\dfrac{\Lambda(n)\,w_n}{(\log n)n^{\frac{1}{2}+it}} + O\big((\log T)^{\lambda}\big),
	\end{split}		
	\end{align} 
where $\Lambda(n)$ is the von-Mangoldt function and $w_n=\max\{0,2\lambda \log_2T-|\log_2T-\log n|\}$. We emphasize that on the left-hand side of \eqref{11_12_18:33pm}, the kernel 
$$
\bigg(\dfrac{\sin (\lambda u\log _2T)}{u}\bigg)^2 ( \pm 3 - 2\sin(u\log _2T)),
$$
is always positive or always negative. This is one of the differences between our proof and the proof of Bondarenko and Seip \cite{BS} as our kernel allows us to have control of the sign. Then, integrating \eqref{11_12_18:33pm} with the resonator, and considering the imaginary part, we can pick the large positive and negative values of $S(t)$. On the right-hand side, the parameter $\lambda$ gives the necessary control on the length of the Dirichlet polynomial to apply \cite[Lemma 7]{BS}. Also, this parameter introduces another Dirichlet polynomial (see \eqref{9_20_11:23am}) as error term in \eqref{11_12_18:33pm}. The length of our Dirichlet polynomial is an important difference of our convolution formula from that of Montgomery, because we need a long enough Dirichlet polynomial to resonate, while Montgomery used a Dirichlet polynomial of length $c\log_2T$. We highlight that the absence of the sign $\pm$ on the right-hand side of \eqref{11_12_18:33pm} is because it will be absorbed in the error term. A similar situation was studied by Mueller \cite{mueller} who established gaps between sign changes of $S(t)$. 

We also recall that the version of the resonance method of Soundararajan \cite{S} was used by Bui, Lester and Milinovich \cite{BLM} to give a new proof of the omega results of Montgomery \eqref{21_00_31_12}.

We would like to remark that our method can be generalized to a family of $L$-functions (see \cite[Section 4]{Chandee2}) and to the argument function $S(\sigma,t)$ defined in a region close to the critical line (see \cite{Chirre} and \cite{Chirre2}). We would also like to refer to \cite{kamal18} for another application of the resonance method to show $\Omega_\pm$ results.

\section{A new convolution formula}
The following lemma is inspired by a result of Montgomery \cite[Lemma 4]{M}. 
\begin{lemma}  \label{9_20_10:05am}
	Assume the Riemann hypothesis. Let $0<\beta<1$ be a fixed number. Let $\alpha>0$, $H\in \R$, and $T^{\beta}\leq t\leq T\log T$, for $T$ sufficiently large. Then
	\begin{align*}
\int_{-(\log T)^{3}}^{(\log T)^{3}}\log\zeta\bigg(\dfrac{1}{2}+i(t+u)\bigg)&\bigg(\dfrac{\sin\alpha u}{u}\bigg)^2e^{iHu}\du = \dfrac{\pi}{2}\displaystyle\sum_{n=2}^{\infty}\dfrac{\Lambda(n)\,w_n(\alpha,H)}{(\log n)n^{\frac{1}{2}+it}} + O\bigg(\dfrac{e^{2\alpha + |H|}}{(\log T)^3}\bigg),
\end{align*} 
where $\Lambda(n)$ is the von-Mangoldt function defined to be $\log p$ if $n=p^m$ with $p$ a prime number and $m\geq 1$ an integer, and zero otherwise, and $w_n(\alpha,H):=\max\{0,2\alpha-|H-\log n|\}$ for all $n\geq 2$.

\end{lemma} 
\begin{proof}
	By \cite[Lemma 3]{M}, we have the following formula
	$$
	\dfrac{1}{2\pi i}\int_{1-i\infty}^{1+i\infty}n^{-s}\bigg(\dfrac{e^{\alpha s}-e^{-\alpha s}}{s}\bigg)^2e^{Hs}\ds = \max\{0,2\alpha-|H-\log n|\},
	$$
	where $n\geq 2$, $\alpha>0$ and $H\in \R$. Since  the Dirichlet series 
	\begin{align*} 
	\log\zeta\bigg(\dfrac{1}{2} + it+s\bigg)=\displaystyle\sum_{n=2}^\infty \dfrac{\Lambda(n)}{(\log n)\, n^{\frac{1}{2}+it+s}}
	\end{align*}
	converges absolutely for $\re{s}>\frac{1}{2}$, we have
	\begin{align} \label{9_20_9am}
	\begin{split}
	\dfrac{1}{2\pi i}\int_{1-i\infty}^{1+i\infty}\log\zeta\bigg(\dfrac{1}{2}+it+s\bigg)\bigg(\dfrac{e^{\alpha s}-e^{-\alpha s}}{s}\bigg)^2e^{Hs}\ds=\displaystyle\sum_{n=2}^\infty \dfrac{\Lambda(n)\,w_n(\alpha, H)}{(\log n)\, n^{\frac{1}{2}+it}},
	\end{split}
	\end{align} 
	where $w_n(\alpha,H):=\max\{0,2\alpha-|H-\log n|\}$. Since we assume RH (see also \cite[Theorem~33]{tit2}), we can move the line of integration in \eqref{9_20_9am} to lie on the following five paths:
	\begin{align*} 
	&L_1=\{1+iu:(\log T)^{3}\leq u <\infty\}, \hspace{1.33cm} L_2=\{\sigma+i(\log T)^{3}:0\leq \sigma\leq 1\}, \\
	&L_3=\{iu:-(\log T)^{3}\leq u <(\log T)^{3}\}, \hspace{0.8cm} L_4=\{\sigma-i(\log T)^{3}:0\leq \sigma\leq 1\}, \\
	&L_5=\{1+iu:-\infty< u \leq -(\log T)^{3}\}.
	\end{align*}
	For each $1\leq j\leq 5$, we define the integrals
	$$
	I_j=\dfrac{1}{2\pi i}\int_{L_j}\log\zeta\bigg(\dfrac{1}{2}+it+s\bigg)\bigg(\dfrac{e^{\alpha s}-e^{-\alpha s}}{s}\bigg)^2e^{Hs}\ds.
	$$
Then\footnote{The notation $f\ll g$ means that there is a constant $c>0$ such that $f(x)\leq c \,g(x)$.},
	\begin{align*}
	\begin{split} 
	\big|I_1\big|, \big|I_5\big| & \ll \int_{(\log T)^{3}}^{\infty}\bigg|\log\zeta\bigg(\dfrac{1}{2}+1+i(t\pm u)\bigg)\bigg|\bigg|\dfrac{e^{\alpha(1\pm iu)}-e^{-\alpha (1\pm iu)}}{1\pm iu}\Bigg|^2e^{H}\du \\ & \ll e^{2\alpha+H}\int_{(\log T)^{3}}^{\infty}\dfrac{1}{u^2}\du 
	\ll \dfrac{e^{2\alpha+|H|}}{(\log T)^{3}}.
	\end{split}
	\end{align*} 
Further, 	
\begin{align*}
\big|I_2\big|, \big|I_4\big| & \ll \int_{0}^{1}\bigg|\log\zeta\bigg(\dfrac{1}{2}+\sigma + i(t\pm(\log T)^{3})\bigg)\bigg|\bigg|\dfrac{e^{\alpha(\sigma\pm i(\log T)^3)}-e^{-\alpha(\sigma\pm i(\log T)^3)}}{\sigma\pm i(\log T)^3}\bigg|^2e^{\sigma H}\d\sigma  \\
& \ll \dfrac{e^{2\alpha+|H|}}{(\log T)^6}\int_{1/2}^{3/2}|\log\zeta(\sigma + i(t\pm(\log T)^{3}))|\d\sigma \ll \dfrac{e^{2\alpha+|H|}}{(\log T)^{5}},
\end{align*}
where in the last line we have used the following estimate (see \cite[Eq. (2.13)]{T}):
$$
\int_{1/2}^{3/2}|\log\zeta(\sigma +it)|\d\sigma \ll \log t.
$$
Finally, the integral $I_3$ gives us the main term:
	\begin{align*}
	I_3& =\dfrac{1}{2\pi}\int_{-(\log T)^{3}}^{(\log T)^{3}}\log\zeta\bigg(\dfrac{1}{2}+i(t+u)\bigg)\bigg(\dfrac{e^{i\alpha u}-e^{-i\alpha u}}{iu}\bigg)^2e^{iHu}\du\\
	& =\dfrac{2}{\pi}\int_{-(\log T)^{3}}^{(\log T)^{3}}\log\zeta\bigg(\dfrac{1}{2}+i(t+u)\bigg)\bigg(\dfrac{\sin \alpha u}{u}\bigg)^2e^{iHu}\du.
	\end{align*}
	Therefore, combining the above estimates, \eqref{9_20_9am}, and the error terms, we get the desired result.
\end{proof} 

The motivation behind the next proposition is to construct a kernel which is always positive or always negative. Note that the function
$$
x\mapsto \pm 3 - 2\sin x
$$
has a unique sign.

\begin{proposition}  \label{9_21_4:20pm}
		Assume the Riemann hypothesis. Let $0<\beta<1$ be a fixed number. Let $0< \lambda <1/2$ be a fixed parameter and $T^\beta\leq t\leq T\log T$,  where $T$ is sufficiently large. Then 
		\begin{align*}
		\int_{-(\log T)^{3}}^{(\log T)^{3}}\log\zeta\bigg(\dfrac{1}{2}+i(t+u)\bigg)&\bigg(\dfrac{\sin (\lambda u\log _2T)}{u}\bigg)^2( \pm 3 - 2\sin(u\log _2T))\du  \\
		& =  \dfrac{i\pi}{2}\displaystyle\sum_{n=2}^{\infty}\dfrac{\Lambda(n)\,w_n}{(\log n)n^{\frac{1}{2}+it}} + O\big((\log T)^{\lambda}\big),		
		\end{align*} 
		where $w_n:=w_n(\lambda\log_2T,\log_2T)$.
\end{proposition}
\begin{proof} We take $\alpha=\lambda\log_2T$ with $H=0$, $H=\log_2T$ and $H=-\log_2T$ in Lemma \ref{9_20_10:05am} and use the linear combination
$$
 \pm 3- 2\sin((\log_2T)u)= \pm 3\, e^{0} +i\big(e^{i(\log_2T)u}-e^{-i(\log_2T)u}\big)
$$
to get 
\begin{align} \label{9_20_11:23am}
\begin{split} 
\int_{-(\log T)^{3}}^{(\log T)^{3}}&\log\zeta\bigg(\dfrac{1}{2}+i(t+u)\bigg)\bigg(\dfrac{\sin (\lambda u\log _2T)}{u}\bigg)^2(\pm 3-2\sin(u\log _2T))\du   \\
& =\pm \dfrac{3\pi}{2}\displaystyle\sum_{n=2}^{\infty}\dfrac{\Lambda(n)\,w_n(\lambda\log_2T,0)}{(\log n)n^{\frac{1}{2}+it}}+  \dfrac{i\pi}{2}\displaystyle\sum_{n=2}^{\infty}\dfrac{\Lambda(n)\,w_n(\lambda\log_2T,\log_2T)}{(\log n)n^{\frac{1}{2}+it}} + O(1).
\end{split} 
\end{align}
Here we have used that $w_n(\lambda\log_2T,-\log_2T)=0$ for $n\geq 2$. 
Also note that the first sum runs over $n \leq (\log T)^{2\lambda}$ and the second sum runs over $(\log T)^{1-2\lambda}\leq n \leq (\log T)^{1+2\lambda}$. We want to bound the first sum. Since $w_n(\lambda\log_2T,0)\ll \log_2T$, using integration by parts and the prime number theorem, we get
\begin{align*}
\displaystyle\sum_{n=2}^{\infty}\dfrac{\Lambda(n)w_n(\lambda\log_2T,0)}{(\log n)n^{\frac{1}{2}+it}} &\ll \log_2T\displaystyle\sum_{2\leq n\leq (\log T)^{2\lambda}}\dfrac{\Lambda(n)}{(\log n)\sqrt{n}}  \ll (\log T)^{\lambda}.
\end{align*} 
Inserting this into \eqref{9_20_11:23am}, we obtain the desired result.
\end{proof}

\section{The Resonator} \label{22_58}
We will use the resonator constructed by Bondarenko and Seip in \cite[Section 3]{BS} (see also \cite[Section 3]{Chirre}). The resonator is a function of the form $|R(t)|^2$, where
\begin{align*} 
R(t)=\displaystyle\sum_{m\in\mathcal{M}'}\dfrac{r(m)}{m^{it}},
\end{align*}
and $\mathcal{M}'$ is a suitable finite set of positive integers whose construction is given below. We start fixing the real number $0<\beta<1$ and define $\kappa=(1-\beta)/2$. Note that $\kappa+\beta<1$. For $T$ sufficiently large, we define $N=[T^{\kappa}]$. Let $\mathcal{P}$ be the set of prime numbers $p$ such that
\begin{align} \label{20_4_6:46pm}
e\log N\log_{2}N < p \leq \exp\big((\log_2N)^{1/8}\big)\log N\log_2N.
\end{align}
We define $f(n)$ to be the multiplicative function supported on the set of square-free numbers such that
$$
f(p):=\sqrt{\dfrac{\log N\log_2N}{\log_3N}}\dfrac{1}{\sqrt{p}\,(\log p-\log_2N-\log_3N)}
$$
for $p\in \mathcal{P}$ and $f(p)=0$ otherwise. For each $k\in\big\{1, ...,\big[(\log_2N)^{1/8}\big]\big\}$, we define the sets:
\begin{align*} 
P_k:=\big\{p: \mbox{prime number such that} \hspace{0.1cm} e^k\log N\log_2N<p\leq e^{k+1}\log N\log_2N\big\},
\end{align*}
\begin{align*}
M_k:=\bigg\{n\in\supp(f): n  \hspace{0.1cm} \mbox{has at least} \hspace{0.1cm} \frac{3\log N}{k^2\log_3N} \hspace{0.1cm} \mbox{prime divisors in}  \hspace{0.1cm} P_k\bigg\},
\end{align*}
and
\begin{align*}
\mathcal{M}:=\supp(f) \backslash \bigcup_{k=1}^{[(\log_2N)^{1/8}]}M_k .
\end{align*}
Now, let $\mathcal{J}$ be the set of integers $j$ such that
$$
\Big[\big(1+T^{-1}\big)^{j},\big(1+T^{-1}\big)^{j+1}\Big)\bigcap \mathcal{M} \neq \emptyset,
$$
and we define $m_j$ to be the minimum of $\big[(1+T^{-1})^{j},(1+T^{-1})^{j+1}\big)\cap \mathcal{M}$ for $j$ in $\mathcal{J}$. Consider the set
$$
\mathcal{M}':=\{m_j:j\in\mathcal{J}\},
$$
and finally we define
$$
r(m_j):=\Bigg(\displaystyle\sum_{n\in\mathcal{M},(1-T^{-1})^{j-1}\leq n \leq (1+T^{-1})^{j+2}}f(n)^2\Bigg)^{1/2}
$$
for every $m_j\in\mathcal{M}'$.

\subsection{Estimates with the resonator} We collect some results related to the resonator which was proved in \cite[Section 3]{BS}. Let us write $\Phi(t)=e^{-t^2/2}$. 

\begin{proposition} \label{19_4_6:04pm}
We have the following properties:
\begin{enumerate}[(i)]
	\item $|R(t)|^2 \leq R(0)^2\ll T^{\kappa}\displaystyle\sum_{l\in\mathcal{M}}f(l)^2$,
	\item $
	\int_{-\infty}^{\infty}|R(t)|^{2}\,\Phi\bigg(\dfrac{t}{T}\bigg)\,\dt \ll T\displaystyle\sum_{l\in\mathcal{M}}f(l)^2.
	$
\end{enumerate}
\end{proposition}
\begin{proof}
$(i)$ and $(ii)$ follows from the definition of $\mathcal{M}'$ and \cite[Lemma 5]{BS}.
\end{proof}

\begin{lemma} \label{20_41_1_20}
	If
	$$
	G(t)=\displaystyle\sum_{n=2}^\infty\dfrac{\Lambda(n)\,w_n}{(\log n)n^{\frac{1}{2}+it}}
	$$	
is absolutely convergent and $w_n\geq 0$ for $n\geq 2$, then
$$	
		\int_{-\infty}^{\infty}G(t)|R(t)|^{2}\,\Phi\bigg(\dfrac{t}{T}\bigg)\,\dt \gg T\sqrt{\dfrac{\log T\log_3T}{\log_2T}}\bigg(\min_{p\in\mathcal{P}}w_p\bigg)\displaystyle\sum_{l\in\mathcal{M}}f(l)^2.
		$$
\end{lemma}
\begin{proof}
See \cite[Lemma 7]{BS}.
\end{proof}

\section{Proof of Theorem \ref{7:03pm_11_12}}
Assume the Riemann hypothesis and consider the parameters defined in Proposition \ref{9_21_4:20pm} and Section \ref{22_58}. We start integrating our convolution formula in Proposition \ref{9_21_4:20pm} in the range $T^\beta\leq t\leq T\log T$  with the factor $|R(t)|^2\Phi(t/T)$. Using $(ii)$ of Proposition \ref{19_4_6:04pm} we get,
\begin{align*}
\int_{T^\beta}^{T\log T}& \Bigg(\int_{-(\log T)^{3}}^{(\log T)^{3}}\log\zeta\bigg(\dfrac{1}{2}+i(t+u)\bigg)\bigg(\dfrac{\sin (\lambda u\log _2T)}{u}\bigg)^2(\pm 3 - 2\sin(u\log _2T))\du\Bigg)|R(t)|^2\Phi\bigg(\dfrac{t}{T}\bigg)\dt  \\
& =   \dfrac{i\pi}{2}\displaystyle\sum_{n=2}^{\infty}\dfrac{\Lambda(n)\,w_n}{(\log n)\sqrt{n}}\Bigg(\int_{T^\beta}^{T\log T}n^{-it}|R(t)|^2\Phi\bigg(\dfrac{t}{T}\bigg)\dt\Bigg) + O\bigg(T(\log T)^{\lambda}\displaystyle\sum_{l\in\mathcal{M}}f(l)^2\bigg).
\end{align*}
Taking the imaginary part, we obtain
\begin{align} \label{12_10_11:43am}
\begin{split}
\int_{T^\beta}^{T\log T}& \Bigg(\int_{-(\log T)^{3}}^{(\log T)^{3}}S(t+u)\bigg(\dfrac{\sin (\lambda u\log _2T)}{u}\bigg)^2(\pm 3 - 2\sin(u\log _2T))\du\Bigg)|R(t)|^2\Phi\bigg(\dfrac{t}{T}\bigg)\dt \\
& = \dfrac{1}{2}\displaystyle\sum_{n=2}^{\infty}\dfrac{\Lambda(n)\,w_n}{(\log n)\sqrt{n}}\,\re\Bigg\{\int_{T^\beta}^{T\log T}n^{-it}|R(t)|^2\Phi\bigg(\dfrac{t}{T}\bigg)\dt\Bigg\} + O\bigg(T(\log T)^{\lambda}\displaystyle\sum_{l\in\mathcal{M}}f(l)^2\bigg).
\end{split}
\end{align}
We denote the above expression by $I_1=I_2$.

\smallskip

\noindent(i)\,\textbf{Analysis of $I_2$}. First we want to complete the integrals involving $\d t$, from $-\infty$ to $\infty$, and then we need to calculate the error terms. Recalling that the sum that appears above runs over $(\log T)^{1-2\lambda}\leq n \leq (\log T)^{1+2\lambda}$, and using the bounds $w_n \ll \log_2T$, $\Phi(t)\leq 1$, the prime number theorem and $(i)$ of Proposition \ref{19_4_6:04pm}, we have
\begin{align*} 
\Bigg|\displaystyle\sum_{n=2}^{\infty}\dfrac{\Lambda(n)\,w_n}{(\log n)\sqrt{n}}\,\re\bigg\{\int_{0}^{T^\beta}{n^{-it}}|R(t)|^2\Phi\bigg(\dfrac{t}{T}\bigg)\dt\bigg\}\Bigg| &\ll \log_2T \displaystyle\sum_{n\leq (\log T)^{1+2\lambda}}\dfrac{\Lambda(n)}{(\log n)\sqrt{n}}\bigg(\int_{0}^{T^\beta}|R(t)|^2\Phi\bigg(\dfrac{t}{T}\bigg)\dt\bigg) \\
& \ll T^{\kappa+\beta}(\log T)^{1/2+\lambda} \displaystyle\sum_{l\in\mathcal{M}}f(l)^2 \ll T\displaystyle\sum_{l\in\mathcal{M}}f(l)^2.
\end{align*}
Similarly, using the rapid decay of $\Phi(t)$, we obtain 
\begin{align*} 
\Bigg|\displaystyle\sum_{n=2}^{\infty}\dfrac{\Lambda(n)\,w_n}{(\log n)\sqrt{n}}&\,\re\bigg\{\int_{T\log T}^{\infty}{n^{-it}}|R(t)|^2\Phi\bigg(\dfrac{t}{T}\bigg)\dt\bigg\}\Bigg| \ll T\displaystyle\sum_{l\in\mathcal{M}}f(l)^2.
\end{align*}
Therefore, we rewrite the integral of $I_2$ from $0$ to $\infty$. Using the fact that $|R(t)|^2$ and $\Phi(t)$ are real and even functions, we get that
$$
\re\Bigg\{\int_{0}^{\infty}{n^{-it}}|R(t)|^2\Phi\bigg(\dfrac{t}{T}\bigg)\dt\Bigg\}=\dfrac{1}{2}\int_{-\infty}^{\infty}{n^{-it}}|R(t)|^2\Phi\bigg(\dfrac{t}{T}\bigg)\dt.
$$
Then
\begin{align} \label{12_10_12:08pm}
I_2 = \dfrac{1}{4}\displaystyle\sum_{n=2}^{\infty}\dfrac{\Lambda(n)\,w_n}{(\log n)\sqrt{n}}\bigg(\int_{-\infty}^{\infty}{n^{-it}}|R(t)|^2\Phi\bigg(\dfrac{t}{T}\bigg)\dt\bigg) + O\bigg(T(\log T)^{\lambda}\displaystyle\sum_{l\in\mathcal{M}}f(l)^2\bigg).
\end{align}
Note that for a fixed $0<\lambda'<\lambda$ and  $(\log T)^{1-2\lambda'}\leq  n \leq (\log T)^{1+2\lambda'}$, we have $w_n\gg\log n$. Using  \eqref{20_4_6:46pm}, we have that each $p\in \mathcal{P}$ satisfies $(\log T)^{1-2\lambda'}\leq p \leq (\log T)^{1+2\lambda'}$, where $T$ is sufficienty large. Therefore 
$$
\min_{p\in\mathcal{P}}w_p\gg \min_{p\in\mathcal{P}}\log p \gg \log_2T.
$$
Using Lemma \ref{20_41_1_20}, we get
\begin{align*} 
\displaystyle\sum_{n=2}^{\infty}\dfrac{\Lambda(n)\,w_n}{(\log n)\sqrt{n}}\bigg\{\int_{-\infty}^{\infty}{n^{-it}}|R(t)|^2\Phi\bigg(\dfrac{t}{T}\bigg)\dt\bigg\} \gg T\sqrt{\log T\log_2T\log_3T}\displaystyle\sum_{l\in\mathcal{M}}f(l)^2.
\end{align*}
Finally, accommodating this lower bound in \eqref{12_10_12:08pm} and using the fact that $\lambda<1/2$, we conclude that
\begin{align} \label{12_10_12:47pm}
I_2 \gg T\sqrt{\log T\log_2T\log_3T}\,\displaystyle\sum_{l\in\mathcal{M}}f(l)^2.
\end{align}

\noindent(ii)\,\textbf{Analysis of $I_1$}. Grouping the signs appropriately, we get
\begin{align*}
I_1& = \int_{T^\beta}^{T\log T}\Bigg(\int_{-(\log T)^{3}}^{(\log T)^{3}}\pm S(t+u)\bigg(\dfrac{\sin (\lambda u\log _2T)}{u}\bigg)^2( 3 \mp 2\sin(u\log _2T))\du\Bigg)|R(t)|^2\Phi\bigg(\dfrac{t}{T}\bigg)\dt \\
&   \leq  \Bigg(\displaystyle\max_{\frac{T^\beta}{2}\leq t\leq 2T\log T} \pm S(t)\Bigg)\Bigg(\int_{-(\log T)^{3}}^{(\log T)^{3}}\bigg(\dfrac{\sin (\lambda u\log _2T)}{u}\bigg)^2( 3 \mp 2\sin(u\log _2T))\du\Bigg)\int_{T^\beta}^{T\log T}|R(t)|^2\Phi\bigg(\dfrac{t}{T}\bigg)\dt. 
\end{align*}
Using \eqref{12_10_12:47pm}, we have that each factor in the above expression is positive. Then by $(ii)$ of Proposition \ref{19_4_6:04pm}, it follows
\begin{align}
 \label{12_10_12:56pm}
I_1 & \leq \Bigg(\displaystyle\max_{\frac{T^\beta}{2}\leq t\leq 2T\log T} \pm S(t)\Bigg)\Bigg(\int_{-(\log T)^{3}}^{(\log T)^{3}}\bigg(\dfrac{\sin (\lambda u\log _2T)}{u}\bigg)^2( 3 \mp 2\sin(u\log _2T))\du\Bigg)\int_{-\infty}^{\infty}|R(t)|^2\Phi\bigg(\dfrac{t}{T}\bigg)\dt \nonumber \\
& \ll \Bigg(\displaystyle\max_{\frac{T^\beta}{2}\leq t\leq 2T\log T} \pm  S(t)\Bigg)(\log_2T)\Bigg(\int_{-\infty}^{\infty}\bigg(\dfrac{\sin (\lambda u)}{u}\bigg)^2( 3 \mp 2\sin u)\du\Bigg)T\displaystyle\sum_{l\in\mathcal{M}}f(l)^2  \\
&  \ll \Bigg(\displaystyle\max_{\frac{T^\beta}{2}\leq t\leq 2T\log T} \pm  S(t)\Bigg)T \log_2T\displaystyle\sum_{l\in\mathcal{M}}f(l)^2. \nonumber
\end{align}
\\
\noindent(iii)\,\textbf{Final analysis}.
Finally, combining \eqref{12_10_11:43am}, \eqref{12_10_12:47pm} and \eqref{12_10_12:56pm}, we conclude that
$$
\sqrt{\dfrac{\log T\log\log\log T}{\log\log T}} \ll \displaystyle\max_{\frac{T^\beta}{2}\leq t\leq 2T\log T} \pm  S(t). 
$$
We obtain the desired restriction $T^\beta\leq T \leq T$ after a trivial adjustment, changing $T$ to $T/(2\log T)$ and making $\beta$ slightly smaller.

\medskip

\section*{Acknowledgements}
We would like to thank Andriy Bondarenko, Micah B. Milinovich, Eero Saksman and Kristian Seip for many valuable discussions and for their insightful comments. We would also like to thank the anonymous referee for the review.

\end{document}